\documentclass[12pt,twoside]{amsart}
\usepackage{amssymb}
\usepackage{amscd}

\title[Infinite dimensional excellent rings]
{Infinite dimensional excellent rings} 
\author{Hiromu Tanaka} 
\address{Graduate School of Mathematical Sciences, 
The University of Tokyo, 
3-8-1 Komaba, Meguro-ku, Tokyo 153-8914, JAPAN} 
\email{tanaka@ms.u-tokyo.ac.jp}

\newcommand{\Reg}[0]{{\operatorname{Reg}}}

\newcommand{\Spec}[0]{{\operatorname{Spec}}}

\newtheorem{thm}{Theorem}[section]
\newtheorem{lem}[thm]{Lemma}

\newtheorem{prop}[thm]{Proposition}

\newtheorem{claim}[thm]{Claim}

\theoremstyle{definition}

\newtheorem{dfn}[thm]{Definition}

\newtheorem{rem}[thm]{Remark}
      
\newtheorem{nota}[thm]{Notation}

\newcommand{\p}{\mathfrak{p}}
\newcommand{\q}{\mathfrak{q}}
\newcommand{\m}{\mathfrak{m}}
\newcommand{\n}{\mathfrak{n}}

\newcommand{\Z}{\mathbb{Z}}

\begin{document}

\maketitle

\begin{abstract}
In this note, we prove that there exist infinite dimensional excellent rings. 
\end{abstract}

\tableofcontents

\section{Introduction}

One of important concepts in commutative ring theory is the notion of noetherian rings. 
On the other hand, it turns out that some pathological phenomena could happen for noetherian rings 
(cf. \cite[Appendix]{Nag62}). 
To remedy the situation, Grothendieck introduced a more restrictive but well-behaved 
class of rings, so-called excellent rings \cite[Ch. IV, \S 7]{Gro65}. 
The purpose of this note is to prove the following theorem. 

\begin{thm}[cf. Theorem \ref{t-main2}]\label{t-main1}
There exist excellent rings whose dimensions are infinite. 
\end{thm}

The construction itself is not new. 
Indeed, we shall show that 
some of the infinite dimensional noetherian rings constructed by Nagata 
are excellent (Theorem \ref{t-main2}). 
Let $B$ be the ring appearing in Notation \ref{n-1}. 
As proved by Nagata, $B$ is an infinite dimensional noetherian ring. 
Thus it suffices to show that $B$ is excellent. 
To this end, we first check that any local ring of $B$ is essentially of finite type over a field. 
This immediately implies that $B$ is a universally catenary G-ring, 
i.e. $B$ satisfies the properties (2) and (3) of Definition \ref{d-exc}. 
Thus what is remaining is to prove that $B$ satisfies the J2 property, i.e. (4) of Definition \ref{d-exc}. 
One of key results is that for any ring $R$, 
the ring homomorphism $R \to R[x_1, x_2, \cdots]$ to a polynomial ring with infinitely many variables is formally smooth. 
Indeed, this result enables us to compare a polynomial ring of finitely many variables and one of infinitely many variables.

\medskip

\textbf{Acknowledgement:} 
The author would like to thank K. Kurano and S. Takagi for answering questions. 
He is also grateful to the referee for reading the paper carefully and for many useful comments.
The author was funded by EPSRC. 

\section{Preliminaries}

\subsection{Notation}

Throughout this paper, all rings are assumed to be commutative and 
to have unity elements for multiplications. 

For a noetherian ring $A$, the {\em regular locus} $\Reg\,(A)$ of $A$ is defined by 
$${\rm Reg}\,(A):=\{\p \in \Spec\,A\,|\,A_{\p}\text{ is regular}\}.$$
A ring homomorphism $\varphi:A \to B$ of noetherian rings is {\em regular} if 
for any prime ideal $\p$ of $A$ and any field extension $A_{\p}/\p A_{\p} \subset L$ of finite degree, 
the noetherian ring $B \otimes_A L$ is regular.

\begin{dfn}\label{d-exc}
A ring $A$ is {\em excellent} if 
\begin{enumerate}
\item 
$A$ is a noetherian ring, 
\item 
$A$ is universally catenary, 
\item 
$A$ is a G-ring, i.e. for any prime ideal $\p$ of $A$, 
the $\p A_{\p}$-adic completion $A_{\p} \to \widehat{A_{\p}}$ is regular,  and 
\item 
$A$ is J2, 
i.e. for any finitely generated $A$-algebra $B$, 
the regular locus ${\rm Reg}\,(B)$ of $B$ is an open subset of $\Spec\,B$. 
\end{enumerate}
\end{dfn}

For some fundamental properties of excellent rings, we refer to 
\cite[Ch. IV, \S 7]{Gro65}, \cite[Section 34]{Mat80} and \cite[\S 32]{Mat89}.

\medskip

A ring homomorphism $\varphi:A \to B$ is a {\em localisation} 
if there exist a multiplicative subset $S$ of $A$ and ring homomorphisms
$$\varphi:A \xrightarrow{\alpha} S^{-1}A \xrightarrow{\theta, \simeq} B$$
where $\alpha$ is the ring homomorphism induced by $S$ and $\theta$ is an isomorphism of rings. 
In this case, $\varphi$ is called the {\em localisation induced by} $S$.

A ring homomorphism $\varphi:A \to B$ is {\em essentially of finite type} 
if there are ring homomorphisms 
$$\varphi:A \xrightarrow{\varphi_1} B' \xrightarrow{\varphi_2} B$$
such that $\varphi_1$ is of finite type and $\varphi_2$ is a localisation.

\begin{rem}\label{r-eft}
If $\varphi:A \to B$ and $\psi:B \to C$ be ring homomorphisms essentially of finite type, 
then also the composite ring homomorphism $\psi \circ \varphi:A \to C$ is essentially of finite type. 
\end{rem}

\subsection{Formally smooth homomorphisms}

In this subsection, we summarise some fundamental properties of formally smooth homomorphisms. 
Let us start by recalling the definition, which is extracted from \cite[Ch. 0, Definition 19.3.1]{Gro64}. 

\begin{dfn}\label{d-fsm}
A ring homomorphism $\varphi:A \to B$ is {\em formally smooth} 
if, for any $A$-algebra $C$, any surjective ring homomorphism $\pi:C \to D$ whose kernel is a nilpotent ideal of $C$ 
and any $A$-algebra homomorphism $u:B \to D$, 
there exist an $A$-algebra homomorphism $v:B \to C$ such that $u=\pi \circ v$. 
\end{dfn}

\begin{rem}
\begin{enumerate}
\item The original definition (\cite[Ch. 0, Definition 19.3.1]{Gro64}) 
treats topological rings, 
whilst we restrict ourselves to considering the rings equipped with the discrete topologies. 
\item 
In \cite[\S 25]{Mat89}, formally smooth homomorphisms are called 0-smooth. 
\end{enumerate}
\end{rem}

\begin{lem}\label{l-fsm-basic}
The following assertions hold. 
\begin{enumerate}
\item 
If $\varphi:A \to B$ and $\psi:B \to C$ are formally smooth ring homomorphisms, 
then also $\psi\circ \varphi:A \to C$ is formally smooth. 
\item 
Let $\varphi:A \to B$ and $A \to A'$ be ring homomorphisms. 
If $\varphi$ is formally smooth, then 
also the induced ring homomorphism $A' \to B\otimes_A A'$ is formally smooth. 
\item 
Let $A$ be a ring and let $S$ be a multiplicatively closed subset of $A$. 
Then the induced ring homomorphism $A \to S^{-1}A$ is formally smooth. 
\end{enumerate}
\end{lem}

\begin{proof}
The assertions (1) and (2) follow from 
\cite[Ch. IV, Remarques 17.1.2(i) and Proposition 17.1.3(ii)(iii)]{Gro67}. 

We prove (3). 
Take a commutative diagram of ring homomorphisms: 
$$\begin{CD}
S^{-1}A @>u>> D=C/I\\
@AA\varphi A @AA\pi A\\
A @>\psi >> C, 
\end{CD}$$
where $I$ is an ideal of $C$ 
satisfying $I^n=0$ for some $n \in \Z_{>0}$, 
and both $\varphi$ and $\pi$ are the induced ring homomorphisms. 
It suffices to prove that 
there exists a ring homomorphism $v:S^{-1}A \to C$ commuting with 
the maps appearing in the above diagram. 
Take $s \in S$. 
Then it holds that 
$$\pi(\psi(s))=u(\varphi(s)) \in (C/I)^{\times}.$$
In particular, there are elements $c \in C$ and $x \in I$ such that 
$$\psi(s)c=1-x.$$ 
It follows from $I^n=0$ that 
$$\psi(s)c(1+x+\cdots+x^{n-1})=(1-x)(1+x+\cdots+x^{n-1})=1.$$
Thus, we get $\psi(s) \in C^{\times}$. 
Hence, there exists a ring homomorphism $v:S^{-1}A \to C$ such that $\psi=v \circ \varphi$. 
Also the equality $u=\pi \circ v$ holds, 
since $\psi=v \circ \varphi$ implies 
$u \circ \varphi=\pi \circ v \circ \varphi$ (cf. \cite[Proposition 3.1]{AM69}). 
Thus (3) holds. 
\end{proof}

\begin{lem}\label{l-f-sm-poly}
Let $R$ be a ring and let $R[x_1, x_2, \cdots]$ be a polynomial ring over $R$ with countably many variables. 
Then the induced ring homomorphism $j: R \to R[x_1, x_2, \cdots]$ is flat and formally smooth. 
\end{lem}

\begin{proof}
Since $R[x_1, x_2, \cdots]$ is a free $R$-module, it is also flat over $R$. 
It follows from \cite[Ch. 0, Corollaire 19.3.3]{Gro64} that $j$ is formally smooth. 
\end{proof}

\begin{lem}\label{l-fsm-perfect}
Let $k_0$ be a perfect field and let $K$ be a field containing $k_0$. 
Let $R$ be a local ring essentially of finite type over $K$. 
Then the following are equivalent. 
\begin{enumerate}
\item 
$R$ is a regular local ring. 
\item 
$R$ is formally smooth over $k_0$. 
\end{enumerate}
\end{lem}

\begin{proof}
Since $k_0$ is perfect, 
the field extension $k_0 \subset K$ is a (possibly non-algebraic) separable extension. 

Assume (1). 
Since $k_0$ is a perfect field, we have that $R$ is geometrically regular over $k_0$ 
in the sense of \cite[Ch. IV, D\'efinition 6.7.6]{Gro65}. 
It follows from \cite[Ch. 0, Proposition 22.6.7(i)]{Gro64} that 
$R$ is formally smooth over $k_0$. 
Thus (2) holds. 

Assume (2). 
Then \cite[Ch. 0, Proposition 22.6.7(i)]{Gro64} implies that $R$ is geometrically regular over $k_0$ 
(cf. \cite[Ch. IV, D\'efinition 6.7.6]{Gro65}). 
In particular, $R$ is regular. 
Thus (1) holds. 
\end{proof}

\section{Proof of the main theorem}

In this section, we prove the main theorem of this paper (Theorem \ref{t-main2}). 
We start by summarising notation used in this section (Notation \ref{n-1}). 
Our goal is to show that the ring $B$ appearing in Notation \ref{n-1} is an excellent ring 
of infinite dimension. 
The ring $B$ is a special case of \cite[Example 1 in Appendix]{Nag62}. 
In particular, it has been already known that $B$ is a noetherian ring of infinite dimension 
(Theorem \ref{t-Nagata}, Proposition \ref{p-local-rings2}). 
Although Lemma \ref{l-avoidance} and Lemma \ref{l-local-rings} are probably used 
in \cite[Example 1 in Appendix]{Nag62} implicitly, 
we give proofs of them for the sake of completeness. 
Indeed, these lemmas also yield the result that 
any local ring of $B$ is essentially of finite type over a field (Proposition \ref{p-local-rings2}), 
which immediately implies that $B$ is a universally catenary G-ring 
(Proposition \ref{p-Gring}, Proposition \ref{p-UC}). 
Then what is remaining is the J2 property, which is settled in Proposition \ref{p-J2}.

\begin{nota}\label{n-1}
\begin{enumerate}
\item Let $k$ be a field. 
\item 
Let 
$$A:=k[x^{(1)}_1, x^{(2)}_1, x^{(2)}_2, x^{(3)}_1, x^{(3)}_2, x^{(3)}_3, \cdots, 
x^{(n)}_1, \cdots, x^{(n)}_n, \cdots].$$
be the polynomial ring over $k$ 
with infinitely many variables $x^{(1)}_1, x^{(2)}_1, x^{(2)}_2, \cdots$. 
For any subset $\mathcal N$ of $\Z_{>0}$, 
we set 
$$A(\mathcal N):=k\left[\bigcup_{n \in \mathcal N}\{x^{(n)}_1, \cdots, x^{(n)}_{n}\}\right] \subset A.$$
In particular, $A=A(\Z_{>0})$. 
For any $n \in \Z_{>0}$, we set 
$$A(n):=A(\{n\})=k[x^{(n)}_1, \cdots, x^{(n)}_n].$$
\item 
For any $n \in \Z_{>0}$, 
we set $\m_{n}:=A(n)x^{(n)}_1+\cdots+A(n)x^{(n)}_1$, 
which is a maximal ideal of $A(n)$. 
It is clear that $\m_nA$ is a prime ideal of $A$. 
\item 
We set 
$$S:=A \setminus \bigcup_{n \in \Z_{>0}} \m_n A
=\bigcap_{n \in \Z_{>0}}(A \setminus \m_n A).$$
Since each $A \setminus \m_n A$ is a multiplicative subset of $A$, so is $S$. 
Let 
$$B:=S^{-1}A.$$
\end{enumerate}
\end{nota}

\begin{lem}\label{l-2-localisation}
Let $A$ be a ring. 
Let 
$$\varphi:A \to S^{-1}A, \quad \psi:A \to T^{-1}A$$
be the ring homomorphisms induced by multiplicative subsets $S$ and $T$ of $A$. 
Assume that there exists a ring homomorphism $\alpha:S^{-1}A \to T^{-1}A$ 
such that $\alpha \circ \varphi=\psi$. 
Set 
$$U:=\{\zeta \in S^{-1}A\,|\, \alpha(\zeta) \in (T^{-1}A)^{\times}\}.$$
Then $\alpha$ is the localisation induced by $U$, i.e. 
there exists a commutative diagram of ring homomorphisms: 
$$\begin{CD}
A @>\varphi >> S^{-1}A\\
@VV\psi V @VV\rho V\\
T^{-1}A @<\theta, \,\,\simeq << U^{-1}(S^{-1}A)
\end{CD}$$
such that $\rho$ is the ring homomorphism induced by $U$, $\theta$ is an isomorphism of rings 
and $\alpha=\theta \circ \rho$. 
\end{lem}

\begin{proof}
Fix $a \in A$ and $s \in S$. 
We have that $(a/1)\cdot (s/1)^{-1}=\alpha(a/s)=b/t_1$ for some $b \in A$ and $t_1 \in T$. 
Then we obtain the equation $t_2t_1a=t_2sb$ in $A$ for some $t_2 \in T$. 
Hence, the equation $(t_2t_1a)/s=(t_2b)/1$ holds in $S^{-1}A$. 
To summarise, for any $\zeta:=a/s \in S^{-1}A$, there exists $t:=t_2t_1 \in T$ such that $\varphi(t)\zeta \in \varphi(A)$. 
Then  the assertion holds 
by applying \cite[Theorem 4.3]{Mat89} for $f:=\psi, g:=\varphi$ and $h:=\alpha$. 
\end{proof}

\begin{lem}\label{l-avoidance}
We use Notation \ref{n-1}. 
Let $\mathfrak a$ be an ideal of $A$ such that $\mathfrak a \subset \bigcup_{n=1}^{\infty} \m_n A$. 
Then there exists a positive integer $n$ such that $\mathfrak a \subset \m_n A$. 
\end{lem}

\begin{proof}
If $\mathfrak a=\{0\}$, then there is nothing to show. 
Thus, we may assume that $\mathfrak a \neq \{0\}$. 
For any $r \in \Z_{>0}$, set 
\[
A_r:=A(\{1, \cdots, r\}),\,\,\, \mathfrak a_r:=\mathfrak a \cap A_r
\,\,\,{\rm and} \,\,\,
\Sigma_r:=\{n \in \Z_{>0}\,|\,\mathfrak a_r \subset \m_nA\}.
\]
We now show the following assertions: 
\begin{enumerate}
\item $\Sigma_1 \supset \Sigma_2 \supset \cdots$. 
\item There exists $r_0 \in \Z_{>0}$ such that  $\Sigma_{r_0}$ is a finite set. 
\item For any $r \in \Z_{>0}$, $\Sigma_r$ is a non-empty set. 
\end{enumerate}
The assertion (1) follows from $\mathfrak a_r=\mathfrak a_{r+1} \cap A_r$. 
Let us prove (2). 
Since $\mathfrak a \neq \{0\}$, there exist $r_0 \in \Z_{>0}$ and $f \in \mathfrak a_{r_0} \setminus \{0\}$. 
Then we have that $f \not\in \m_nA$ for any $n \geq r_0+1$, 
which implies $\mathfrak a_{r_0} \not\subset \m_n A$ for any $n \geq r_0+1$. 
In particular,  $\Sigma_{r_0}$ is a finite set, as desired. 
We now show (3). 
Fix $r \in \Z_{>0}$. 
It follows from $\mathfrak a \subset \bigcup_{n=1}^{\infty} \m_n A$ that 
\[
\mathfrak a_r= \mathfrak a \cap A_r 
\subset \left(\bigcup_{n=1}^{\infty} \m_n A\right) \cap A_r
=\bigcup_{n=1}^{\infty} (\m_n A \cap A_r)=\bigcup_{n=1}^{r} (\m_n A \cap A_r).
\]
Thanks to the prime avoidance lemma \cite[Proposition 1.11]{AM69}, 
there exists $n \in \{1, \cdots, r\}$ such that $\mathfrak a_r \subset \m_n A \cap A_r$. 
Hence, (3) holds. 

By (1)--(3), we obtain $\bigcap_{r=1}^{\infty} \Sigma_r \neq \emptyset$. 
Take $n \in \bigcap_{r=1}^{\infty} \Sigma_r$. 
Then it holds that $\mathfrak a_r \subset \m_n A$ for any $r \in \Z_{>0}$. 
Hence we obtain  
\[
\mathfrak a= \bigcup_{r=1}^{\infty} \mathfrak a_r \subset \m_n A, 
\]
as desired. 
\end{proof}

\begin{lem}\label{l-local-rings}
We use Notation \ref{n-1}. 
Take $n \in \Z_{>0}$. 
Then $A_{\m_n A}$ is isomorphic to the local ring 
$K[x_1^{(n)}, \cdots, x_n^{(n)}]_{(x_1^{(n)}, \cdots, x_n^{(n)})}$ 
of the polynomial ring $K[x_1^{(n)}, \cdots, x_n^{(n)}]$ with respect to 
the maximal ideal $(x_1^{(n)}, \cdots, x_n^{(n)})$, 
where $K$ is the quotient field of $A(\Z_{>0} \setminus \{n\})$. 
\end{lem}

\begin{proof}
Let 
$$T:=A(\Z_{>0} \setminus \{n\}) \setminus \{0\}.$$
Then we get $T \subset A \setminus \m_n A$. 
In particular, we obtain a factorisation of ring homomorphisms:  
$$\psi:A \hookrightarrow K[x_1^{(n)}, \cdots, x_n^{(n)}]=T^{-1}A \overset{\varphi}{\hookrightarrow} A_{\m_nA}.$$
It follows from \cite[Theorem 4.3, Corollary 1]{Mat89} that 
$A_{\m_nA}$ is isomorphic to $K[x_1^{(n)}, \cdots, x_n^{(n)}]_{\n}$, 
where $\n:=\varphi^{-1}(\m_n A_{\m_nA})$. 
Since $\n$ is a prime ideal of $K[x_1^{(n)}, \cdots, x_n^{(n)}]$ such that 
$x_1^{(n)}, \cdots, x_n^{(n)} \in \n$, 
it holds that $\n=(x_1^{(n)}, \cdots, x_n^{(n)})$, as desired. 
\end{proof}

\begin{thm}\label{t-Nagata}
We use Notation \ref{n-1}. 
Then $B$ is a noetherian ring. 
\end{thm}

\begin{proof}
It follows from \cite[(E1.1) in page 203]{Nag62} that $B$ is a noetherian ring. 
Note that the altitude of a ring $R$ means the dimension of $R$ (cf. \cite[Section 9]{Nag62}). 
\end{proof}

\begin{prop}\label{p-local-rings2}
We use Notation \ref{n-1}. 
Then the following hold. 
\begin{enumerate}
\item 
For any positive integer $n$, 
there exists a unique maximal ideal $\m'_n$ of $B$ such that $\dim B_{\m'_n}=n$. 
Furthermore, $B_{\m'_n}$ is isomorphic to 
$K[x_1, \cdots, z_n]_{(z_1, \cdots, z_n)}$, 
where $K[z_1, \cdots, z_n]$ is the polynomial ring over a field $K$ 
and $K[x_1, \cdots, z_n]_{(z_1, \cdots, z_n)}$ is the local ring with respect to the maximal ideal $(z_1, \cdots, z_n)$. 
\item 
Any maximal ideal $\n$ of $B$ is equal to $\m'_n$ for some $n \in \Z_{>0}$. 
\item 
For any prime ideal $\p$ of $B$, 
the local ring $B_{\p}$ is essentially of finite type over a field.  
\end{enumerate}
\end{prop}

\begin{proof}
The assertions (1) and (2) follow from Lemma \ref{l-avoidance} and Lemma \ref{l-local-rings}. 
Then (1) and (2) imply (3). 
\end{proof}

\begin{prop}\label{p-Gring}
We use Notation \ref{n-1}. 
Then, for any prime ideal $\p$ of $B$, 
the $\p B_{\p}$-adic completion $B_{\p} \to \widehat{B_{\p}}$ is regular.  
\end{prop}

\begin{proof}
The assertion follows from Proposition \ref{p-local-rings2}(3) (cf. \cite[Ch. IV, Scholie 7.8.3(ii)]{Gro65}). 
\end{proof}

\begin{prop}\label{p-UC}
We use Notation \ref{n-1}. 
Then $B$ is regular. In particular, $B$ is a universally catenary ring. 
\end{prop}

\begin{proof}
It follows from Proposition \ref{p-local-rings2} 
that $B$ is a noetherian regular ring. 
In particular, $B$ is universally catenary by \cite[Theorem 17.8 and Theorem 17.9]{Mat89}. 
\end{proof}

\begin{prop}\label{p-J2}
We use Notation \ref{n-1}. 
Let $D$ be a finitely generated $B$-algebra. 
Then the following hold. 
\begin{enumerate}
\item 
For any prime ideal $\q$ of $D$, the local ring $D_{\q}$ of $D$ with respect to $\q$ is 
essentially of finite type over a field. 
\item 
The regular locus $\Reg(D)$ of $D$ is an open subset of $\Spec\,D$. 
\end{enumerate}

\end{prop}

\begin{proof}
There exists a polynomial ring $C:=B[y_1, \cdots, y_m]$ over $B$ and an ideal $I$ of $C$ such that $D \simeq C/I$. 

Let us show (1). 
Let $\p$ be the pullback of $\q$ to $B$. 
By Proposition \ref{p-local-rings2}(3), 
there exist a field $K$ and a ring homomorphism 
$\zeta_0:K \to B_{\p}$ that is essentially of finite type. 
Consider the ring homomorphisms 
$$\zeta: K \xrightarrow{\zeta_0} B_{\p} \xrightarrow{\zeta_1} C \otimes_B B_{\p} \xrightarrow{\zeta_2} D \otimes_B B_{\p} \xrightarrow{\zeta_3} D_{\q}$$
where $\zeta_1, \zeta_2$ and $\zeta_3$ are the induced ring homomorphisms. 
Then each $\zeta_i$ is essentially of finite type (cf. Lemma \ref{l-2-localisation}). 
Therefore, also the composite ring homomorphism $\zeta$ is essentially of finite type. 
Thus (1) holds.

Let us prove (2). 
We can write 
$$I=\sum_{i=1}^r C f_i, \quad f_i=\sum_{(j_1, \cdots, j_m) \in \Z_{\geq 0}^m} \frac{g_{i, J}}{h_{i, J}} y_1^{j_1}\cdots y_m^{j_m}$$
for $J:=(j_1, \cdots, j_m) \in \Z_{\geq 0}^m$, 
$g_{i, J} \in A$ and $h_{i, J} \in S=\bigcap_{n \in \Z_{>0}} (A \setminus \m_nA)$. 
We can find a finite subset $\mathcal N$ of $\mathbb Z_{>0}$ such that 
$$\bigcup_{i, J} \{g_{i, J}, h_{i, J}\} \subset A(\mathcal N)=:A'.$$
We set 
$$S':=\bigcap_{n \in \mathcal N} (A' \setminus \m_nA'), \quad B':=S'^{-1}A', \quad C':=B'[y_1, \cdots, y_m].$$ 
Since 
$$S \cap A'=\left(\bigcap_{n \in \Z_{>0}} (A \setminus \m_nA)\right) \cap A'
=\bigcap_{n \in \Z_{>0}} (A' \setminus \m_nA)=\bigcap_{n \in \mathcal N} (A' \setminus \m_nA)=S',$$
we have that $h_{i, J} \in S'$. 
In particular, we get $f_i \in C'$. 
We set 
$$I':=\sum_{i=1}^r C' f_i, \quad D':=C'/I'.$$
Thanks to $S' \subset S$, we get a commutative diagram 
$$\begin{CD}
A' @>>> B' @>>> C' @>>> D'\\
@VV\alpha V @VV\beta V @VV\gamma V @VV\delta V\\
A @>>> B @>>> C @>>>D.
\end{CD}$$
Note that the middle and right square are cocartesian and there is a factorisation of ring homomorphisms 
$$\beta: B'=S'^{-1}A' \xrightarrow{S'^{-1}\alpha} S'^{-1}A \xrightarrow{\beta'} B,$$ 
where $\beta'$ is a localisation (cf. Lemma \ref{l-2-localisation}). 

Since $\alpha:A' \to A$ is flat and formally smooth (Lemma \ref{l-f-sm-poly}), 
it follows from Lemma \ref{l-fsm-basic}(2) that also 
$S'^{-1}\alpha:B'\to S'^{-1}A$
is flat and formally smooth. 
Furthermore, by Lemma \ref{l-fsm-basic}(3), 
$S'^{-1}A \xrightarrow{\beta'} B$ is flat and formally smooth. 
Therefore, it holds by Lemma \ref{l-fsm-basic}(1) that $\beta:B' \to B$ is flat and formally smooth. 
In particular, $\gamma:C' \to C$ and $\delta:D' \to D$ are flat and formally smooth by Lemma \ref{l-fsm-basic}(2). 

\begin{claim}\label{c-J2}
Let $\q$ be a prime ideal of $D$ and set $\q':=\delta^{-1}(\q)$. 
Then $D_{\q}$ is regular if and only if $D'_{\q'}$ is regular. 
\end{claim}

\begin{proof}[Proof of Claim \ref{c-J2}]
There are ring homomorphisms 
$$\epsilon:D'_{\q'} \xrightarrow{\delta \otimes_{D'} D'_{\q'} } D \otimes_{D'} D'_{\q'} \xrightarrow{\theta} D_{\q}$$ 
such that $\theta$ is a localisation (Lemma \ref{l-2-localisation}). 
Since $\delta:D' \to D$ is flat and formally smooth, 
both $\delta \otimes_{D'} D'_{\q'}$ and $\theta$ are flat and formally smooth 
(Lemma \ref{l-fsm-basic}(2)(3)). 
Therefore, $\epsilon$ is 
is faithfully flat and formally smooth 
(Lemma \ref{l-fsm-basic}(1) and \cite[Theorem 7.3(ii)]{Mat89}).

Thus, if $D_{\q}$ is regular, then so is $D'_{\q'}$ by \cite[Theorem 23.7]{Mat89}. 
Conversely, assume that $D'_{\q'}$ is regular. 
By Lemma \ref{l-fsm-perfect} and the fact that the prime field $k_0$ contained in $k$ is a perfect field, 
we have that the induced ring homomorphism $k_0 \to D'_{\q}$ is formally smooth. 
Since also $\epsilon:D'_{\q'} \to D_{\q}$ is formally smooth, 
it holds by Lemma \ref{l-fsm-basic}(1) that $k_0 \to D_{\q}$ is formally smooth. 
Since $D_{\q}$ is essentially of finite type over a field by (1), 
it follows again from Lemma \ref{l-fsm-perfect} that $D_{\q}$ is regular. 
This completes the proof of Claim \ref{c-J2}. 
\end{proof}

Let us go back to the proof of Proposition \ref{p-J2}(2). 
Let 
$$\widetilde{\delta}:\Spec\,D \to \Spec\,D'$$
be the continuous map induced by $\delta$. 
It follows from Claim \ref{c-J2} that 
$$\widetilde{\delta}^{-1}(\Reg\,(D'))=\Reg\,D.$$
Since $\Reg\,(D')$ is an open subset of $\Spec\,D'$, 
$\Reg\,D$ is an open subset of $\Spec\,D$, as desired. 
Thus (2) holds.  
This completes the proof of Proposition \ref{p-J2}. 
\end{proof}

\begin{thm}\label{t-main2}
We use Notation \ref{n-1}. 
Then $B$ is a regular excellent ring such that $\dim B=\infty$.  
\end{thm}

\begin{proof}
By Theorem \ref{t-Nagata}, $B$ is a noetherian ring. 
Thanks to Proposition \ref{p-Gring}, Proposition \ref{p-UC} and Proposition \ref{p-J2}, 
we see that $B$ is an excellent ring. 
It follows from Proposition \ref{p-local-rings2}(1)(2) that $B$ is regular and infinite dimensional.  
\end{proof}

\end{document}